\documentclass[12pt]{article}
\usepackage{amsmath}
\usepackage{amsfonts}
\usepackage{graphicx}
\usepackage{xcolor}
\usepackage{latexsym}
\makeatletter \@addtoreset{equation}{section}

\def \one{\ \hbox{I\hskip-.60em 1}}

\begin{document}

\newtheorem{theorem}{Theorem}[section]
\newtheorem{lemma}[theorem]{Lemma}
\newtheorem{cor}[theorem]{Corollary}
\newtheorem{defn}[theorem]{Definition}
\newtheorem{assp}[theorem]{Assumption}
\newtheorem{cond}[theorem]{Condition}
\newtheorem{expl}[theorem]{Example}
\newtheorem{prop}[theorem]{Proposition}
\newtheorem{rmk}[theorem]{Remark}

\newtheorem{remark}{Remark}[section]

\newcommand{\eproof}{\indent\vrule height6pt width4pt depth1pt\hfil\par\medbreak}

\def\proof{\noindent{\it Proof. }}

\title{\itshape A note on Euler approximations for stochastic differential equations with delay}

\author{Istvan Gy\"{o}ngy and Sotirios Sabanis
\thanks{Corresponding author. Email:
s.sabanis@ed.ac.uk \vspace{6pt}} \\
\vspace{6pt} \em{Maxwell Institute for Mathematical Sciences and
School of Mathematics,} \\  \em{University of Edinburgh, Edinburgh
EH9 3JZ, U.K.} }

\maketitle

\begin{abstract}

\bigskip

\noindent An existence and uniqueness theorem for a class of
stochastic delay differential equations is presented, and the
convergence of Euler approximations for these equations is proved
under general conditions. Moreover, the rate of almost sure
convergence is obtained under local Lipschitz and also under
monotonicity conditions.



\bigskip

\end{abstract}

\date{ }

\maketitle

\section{Introduction}

Stochastic delay differential equations (SDDEs) play an important
role in understanding and modelling many real world phenomena for
which the principle of causality does not apply. One could refer to
\cite{Mao-Yuan-Zou}, \cite{Elsanosi-Oksendal-Sulem},
\cite{McWilliams-Sabanis} and \cite{Arriojas_et_al} for applications
in biology, ecology, economics and finance to name a few, without,
of course, exhausting the long list of the existing literature on
the subject matter. It is important therefore to determine precisely
under which conditions one obtains a unique solution for a delay
system and, moreover, to study the convergence of suitable numerical
schemes. To this end, we employ techniques from the theory of
stochastic differential equations (SDEs) with random coefficients so
as to determine the conditions for uniqueness and existence of
solutions of delay models. Furthermore, we investigate the
convergence properties of Euler schemes that are used to approximate
the aforementioned models.

\noindent Strong discrete-time approximations of SDDEs (in
$L^p$-sense) have been studied by several authors, including
\cite{Kuchler-Platen}, \cite{Baker-Buckwar}, \cite{Mao-Sabanis} and
\cite{HU_et_al} amongst others. Moreover, in recent years, new
findings appeared in the direction of week approximations, see for
example \cite{Buckwar-et-al}.

\noindent Our reason for presenting here results on two types of
convergence, almost sure and in probability, for numerical schemes
of delay models is twofold.

\noindent First we contribute to the understanding of delay models
by providing new results while imposing essentially weaker
conditions on the smoothness of the coefficients in comparison with
the current literature, see for example \cite{Baker-Buckwar},
\cite{Buckwar-et-al}, \cite{Kuchler-Platen},
\cite{Kloeden-Neuenkirch} and \cite{Mao-Sabanis} and the references
therein. The convergence of Euler approximations is proved under
local monotonicity condition, which is much weaker than the local
Lipschitz condition that appears in \cite{Mao-Sabanis}. Moreover, no
smoothness condition on the initial data, on the delay function and
on the drift and diffusion coefficients in the delay argument are
assumed in order to obtain the convergence in probability. The main
result of \cite{Mao-Sabanis}, Theorem 2.1, states convergence in
mean square. It should be noted, that under condition $(H_3)$ used
in \cite{Mao-Sabanis}, our convergence results clearly imply
convergence in mean square as well. In addition, under local
monotonicity condition we present the almost sure rate of
convergence of Euler approximations whereas, Theorem 2.5 in
\cite{Kloeden-Neuenkirch} requires global Lipschitzness.

\noindent Second, we facilitate the development of the theory with
regards to the understanding of quantitative and qualitative
characteristics of solutions of delay equations. As an example, one
may consider the study of different types of stability (almost sure
asymptotic, exponential,  mean square etc) for solutions of such
models. This is an area which has attracted significant attention in
recent years, see for example \cite{Mao_et_al},
\cite{Higham_Mao_Yuan}, \cite{Wu-Mao-Szpruch} and the references
therein.

\noindent Finally, we note that although the authors in
\cite{Renesse-Scheutzow} provide an existence and uniqueness theorem
for stochastic functional differential equations, a more general
class of SDDEs than ours, their conditions are stronger in
comparison with Theorem \ref{SDDE theorem} below. The reason for
this is that we do not require any smoothness of the drift and
diffusion coefficients in the arguments corresponding to the delays.

\noindent We conclude this section by introducing some basic
notation. Let $xy$  be the scalar product of vectors $x$, $y \in
\mathbb{R}^d$ and $|x|$ be the length of $x$. Moreover, if $g \in
\mathbb R^{d\times m}$ is a matrix, then let $g^T$ and $|g|$ denote
the transpose of $g$ and the Hilbert–-Schmidt norm respectively,
i.e. $|g|=\sqrt{\mbox{tr}(gg^T)}$. In addition, let $[x]$ denote the
integer part of the real number $x$. Finally, let $\mathcal P$ and
${\mathcal B}(V)$ denote the predictable $\sigma$-algebra on
$\mathbb R_+\times\Omega$ and the $\sigma$-algebra of Borel sets of
topological spaces $V$ respectively.

\section{Main Results}

\noindent Let $\beta(t, y_1, \ldots, y_k, x)$ and $\alpha(t, y_1,
\ldots, y_k, x)$ be ${\mathcal B}(\mathbb R_+)\otimes {\mathcal
B}(\mathbb R^{d\times k}) \otimes {\mathcal B}(\mathbb
R^d)$-measurable functions with values in $\mathbb R^d$ and $\mathbb
R^{d\times m}$, respectively. Consider the stochastic delay
differential equation
\begin{align}\label{SDDE_1}
dX(t) & = \beta(t, X(\delta_1(t)), \ldots, X(\delta_k(t)), X(t))dt +
\alpha(t, X(\delta_1(t)), \ldots, X(\delta_k(t)), X(t))dW_t, \mbox{ } \\
X(t) & = \xi(t), \qquad \forall t\in[-C,0] \nonumber
\end{align}
on a fixed probability space $(\Omega, \mathbb{F}, \mathbb{P})$,
equipped with a right-continuous filtration
$\mathbb{F}:=\{\mathcal{F}_t\}_{t\ge 0}$ and an $m$-dimensional
Wiener martingale $W:=\{W_{t}\}_{t\geq0}$, where
$\xi:=\{\xi(t)\}_{t\in[-C,0]}$ is a continuous process which is
$\mathcal{F}_0$-measurable for every $t\in[-C,0]$ for a fixed
constant $C>0$,  and $\delta_i(t)$ is an increasing function of $t$
such that $-C\le\delta_i(t)\le [\tfrac{t}{\tau}]\tau$ for some
positive constant $\tau$ and $1\le i\le k$. Note that two popular
cases for $\delta_i$  are included here. These are the fixed delay
case, $\delta_i(t)= t-\tau$, and $\delta_i(t)=
[\tfrac{t}{\tau}]\tau$ which appear in many applications, see for
example \cite{Arriojas_et_al}.

\noindent Fix a constant $T>0$. Let $\mathbb L$ denote the set of
nonnegative integrable functions on $[0,T]$ and $y:=(y_1, \ldots,
y_k)$. Consider the following conditions:
\begin{itemize}
\item[($C_1$)] The function $\beta(t, y, x)$
 is continuous in $x$ for any $t$ and $y$.

\item[($C_2$)] For every $R>0$, there exists a $K_R\in \mathbb L$
such that for all $t\in[0,T]$
$$
\sup_{|x|\le R}\sup_{|y|\leq R}|\beta(t, y, x)| \leq K_R(t)
$$
\item[($C_3$)] For every $R>0$, there exists a $L_R\in\mathbb L$
such that
\begin{equation}
2(x-z)(\beta(t, y, x) - \beta(t, y, z))+
|\alpha(t, y, x) - \alpha(t, y, z)|^2
 \le L_R(t)|x-z|^2,  \nonumber
\end{equation}
for $t\in[0,T]$ and  $|x|, |y|, |z| < R$
\item[($C_4$)]  For any $R>0$, there exists a $M_R\in\mathbb L$ such that
\begin{equation}
2x\beta(t, y, x)+ |\alpha(t, y, x)|^2 \le M_R(t)(1+|x|^2)  \nonumber
\end{equation}
for all $t\in[0,T]$, $x\in\mathbb R^d$ and $|y|\leq R$.
\end{itemize}

\begin{remark} \label{remark}
Note that conditions ($C2$) and ($C4$) imply the existence of a
$K_R\in\mathbb L$ such that
$$
\sup_{|x|\le R}\sup_{|y|\leq R}|\alpha(t, y, x)|^2 \leq K_R(t)
$$
for all $t\in[0,T]$.
\end{remark}

\begin{theorem} \label{SDDE theorem}
Let us assume that conditions ($C_1$)-($C_4$) hold, then there
exists a unique process $\{X(t)\}_{t\in[0, T]}$ that satisfies
equation (\ref{SDDE_1}).
\end{theorem}

\noindent For $n\ge 1$, consider the following Euler scheme for
equation (\ref{SDDE_1})
\begin{equation} \label{EulerSDDE}
dX_n(t)= \beta(t, Y_n(t), X_n(\kappa_n(t)))dt + \alpha(t, Y_n(t) ,
X_n(\kappa_n(t)))dW_t,
\end{equation}
for $t\in[0,T]$, where $X_n(t)=\xi(t)$ on $[-C,0]$, $Y_n(t):=
(X_n(\delta_1(t)), \ldots, X_n(\delta_k(t)))$ and
$\kappa_n(t):=[nt]/n$. Note that if ($C_2$) and Remark \ref{remark}
hold, then (2.2) is well-defined.  In addition, consider the
condition below:
\begin{itemize}
\item[($C_5$)] The functions $\beta(t, y, x)$ and $\alpha(t, y,
x)$ are continuous in $y$ uniformly in $x$ from compacts, i.e. for
every $R>0$ and $t\in[0,T]$,
$$
\sup_{|x|\le R}[|\beta(t, y, x) - \beta(t, y^{'}, x)| + |\alpha(t,
y, x) - \alpha(t, y^{'}, x)|] \to 0 \quad \mbox{as  } y\to y^{'}.
$$
\end{itemize}

\begin{theorem}\label{TheoremEulerSDDE}
Let us assume that conditions ($C_1$)-($C_5$) hold. Consider
equation (\ref{SDDE_1}) and the corresponding Euler scheme defined
by (\ref{EulerSDDE}). Then
$$
\sup_{t\le T}|X_n(t)-X(t)| \xrightarrow[]{\mathbb{P}}  0 \qquad
\mbox{as } n\to \infty.
$$
\end{theorem}

\noindent Let $\{X_n\}_{n\ge 1}$  be a sequence of almost surely
finite random variables and $\{a_n\}_{n=1}^{\infty}$ be a positive
numerical sequence. Then
$$
X_n=\mathcal{O}(a_n)
$$
denotes that there exists an almost surely finite random variable
$\zeta$ such that, almost surely,
$$
|X_n| \le \zeta a_n
$$
for any $n\ge 1$. In order to obtain an estimate for the a.s
convergence of the Euler scheme, we consider the following
conditions:
\begin{itemize}
\item[($A_1$)] For every $R>0$, there exists a constant $k_R$
such that, for all $t\in[0,T]$,
\begin{align}
\sup_{|x|\le R}\sup_{|y|\leq R}\Big(|\beta(t, y, x)| + |\alpha(t, y,
x)|\Big) \leq k_R. \nonumber
\end{align}
\item[($A_2$)]
For every $R>0$, there exists a constant $c_R$ such that, for every
$t\in[0,T]$,
\begin{equation}  \label{LipDelay}
|\beta(t,y,x)-\beta(t,y^{'}, x^{'})|\le c_R(|y-y^{'}|+|x-x^{'}|)
\end{equation}
\begin{equation}
|\alpha(t,y,x)-\alpha(t,y^{'},x^{'})|^2\le
c_R(|y-y^{'}|^2+|x-x^{'}|^2) \nonumber
\end{equation}
whenever $|x|,\, |x^{'}|,\,|y|, \, |y^{'}| < R$.
\item[($A_3$)]
For every $R>0$, there exists a constant $c_R$ such that, for every
$t\in[0,T]$,
\begin{align}
2(x-x^{'})\Big(\beta(t,y,x)-\beta(t,y, x^{'})\Big)&\le
c_R|x-x^{'}|^2 \nonumber \\ |\beta(t,y,x)-\beta(t,y^{'}, x)| &\le
c_R|y-y^{'}| \nonumber
\end{align}
whenever $|x|,\, |x^{'}|,\,|y|, \, |y^{'}| < R$.
\end{itemize}

\begin{remark}
Note that conditions ($A1$) and ($A2$) imply that equation
(\ref{SDDE_1}) has a unique (complete) local solution. The same is
true if (\ref{LipDelay}) is replaced by ($A3$).
\end{remark}

\begin{theorem} \label{rateSDDEcase}
Let conditions ($A_1$) and ($A_2$) hold. Assume that equation
(\ref{SDDE_1}) admits a solution $\{X(t)\}_{t\in[0,\,T]}$. Let
$\{X_n(t)\}_{t\in[0,\,T]}$ denote the solution of the Euler scheme
(\ref{EulerSDDE}). Then,
\begin{equation}                                                            \label{estimateSDDE1}
\sup_{t\le T}|X(t) - X_n(t)|=\mathcal{O}(n^{-\gamma}) \quad
\mbox{(a.s.)}
\end{equation}
for every $\gamma<1/2$. Moreover, if one replaces (\ref{LipDelay})
with ($A3$), then (\ref{estimateSDDE1}) holds for every
$\gamma<1/4$.
\end{theorem}

\begin{remark}
Note that without loss of generality, it is assumed henceforth that
$T$ is a multiple of $\tau$. To see this, one considers equation
(\ref{SDDE_1}) for every $t\le T^{'}$, where $T^{'}= N\tau \ge T$
and $N$ is a positive integer, and observes that all the above
conditions are satisfied when $\beta$ and $\alpha$ are replaced by
$\beta \one_{\{t\le T\}}$ and $\alpha\one_{\{t\le T\}}$.
\end{remark}

\section{Existence and Uniqueness}

Let $b(t, x)$ and $\sigma(t, x)$ be $\mathcal P\otimes\mathcal
B(\mathbb R^d)$-measurable functions  with values in $\mathbb R^d$
and $\mathbb R^{d\times m}$ respectively. Let $t_0$ and $t_1$ be any
positive constants such that $0 \le t_0<t_1 \le T$. Let also
$\mathcal A$ denote the set of nonnegative $\mathbb{F}$-adapted
stochastic processes $L=\{L(t)\}_{t\in[0,T]}$ such that
$$
\int_0^T L(t)\,dt<\infty \quad \mbox{(a.s.).}
$$
Consider
\begin{equation}\label{random}
dX(t) = b(t, X(t))\,dt + \sigma(t, X(t))\,dW_t, \mbox{ } \forall \,
t\in [t_0,\, t_1],
\end{equation}
with an initial condition $X(t_0)$ which is an
$\mathcal{F}_{t_0}$-measurable, almost surely finite random
variable. Furthermore, consider conditions
\begin{itemize}
\item[($D_1$)] The function $b(t, x)$ is continuous in $x$ for any $t$ and $\omega$.
\item[($D_2$)] For every $R>0$, there exists $\mathcal{K}_R\in\mathcal A$
such that, almost surely,
\begin{equation}
\sup_{|x|\le R} |b(t,x)| \le \mathcal{K}_R(t) \nonumber
\end{equation}
for any $t\in[t_0,\, t_1]$.
\item[($D_3$)] For every $R>0$, there exists $\mathcal{L}_R\in\mathcal A$
such that, almost surely,
\begin{equation}
2(x-z)(b(t, x) - b(t, z))+ |\sigma(t, x) - \sigma(t, z)|^2
 \le \mathcal{L}_R(t)|x-z|^2 \nonumber
\end{equation}
for any $t\in[t_0,\, t_1]$ and  $|x|,|z| < R$.
\item[($D_4$)] There exists $\mathcal{M}\in\mathcal A$ such
that, almost surely,
\begin{equation}
2xb(t, x)+ |\sigma(t, x)|^2 \le \mathcal{M}(t)(1+|x|^2) \nonumber
\end{equation}
for every $t\in[t_0,\, t_1]$ and $x\in\mathbb{R}^d$.
\end{itemize}

\noindent The following existence and uniqueness theorem is known
from \cite{Krylov_Istvan1980} and \cite{Krylov1990}.

\begin{theorem} \label{SDEwRC_theorem}
Let us assume that conditions ($D_1$)-($D_4$) hold, then there
exists a unique process $\{X(t)\}_{t\in[t_0, t_1]}$ that satisfies
equation (\ref{random}).
\end{theorem}

\noindent We are ready now to proceed with the proof of the main
theorem of this section.

\noindent {\it Proof of Theorem \ref{SDDE theorem}.} One considers
first the interval $[0,\,\tau)$ and observes that this reduces to
the well-known case of stochastic differential equations (without
delay) where the assumptions ($C_1$)-($C_4$) guarantee the existence
of a unique, continuous  solution (see Theorem \ref{SDEwRC_theorem}
above). One then observes that
$$
X(\tau):= X(0) + \int_0^{\tau} \beta(t,Y(t), X(t))dt +
\int_0^{\tau}\alpha(t, Y(t), X(t))dW_t,
$$
with $Y(t):=(X(\delta_1(t)),\ldots, X(\delta_k(t)))$, is well
defined. Inductively, one may assume that a unique, continuous
solution exists on the interval $[(i-1)\tau,\,i\tau]$, for some
positive integer $i \in \{1, \ldots, N\}$, with the aim to prove
that the same is true on $[i\tau,\,(i+1)\tau]$. One then considers
equation (\ref{random}) with
\begin{equation} \label{random_b_sigma}
b(t,x):=\beta(t, Y(t), x), \quad\sigma(t,x):=\alpha(t, Y(t), x)
\end{equation}
for every $t\in [i\tau,\,(i+1)\tau)$, and with initial condition
$X({i\tau})$ which is an $\mathcal{F}_{i\tau}$-measurable, almost
surely finite random variable. One immediately observes that $b(t,
x)$ and $\sigma(t, x)$ are $\mathcal P\otimes\mathcal B(\mathbb
R^d)$-measurable functions with values in $\mathbb R^d$ and $\mathbb
R^{d\times m}$ respectively as a direct consequence of the
measurability properties of the aforementioned functions $\beta$ and
$\alpha$. Furthermore, one obtains that ($D1$)-($D4$) hold for every
$t\in [i\tau,\,(i+1)\tau)$ due to assumptions ($C1$)-($C4$). More
precisely, ($D_1$) is a direct consequence of ($C_1$); ($D_2$) is a
consequence of ($C_2$) since $\sup_{i\tau \le t <
(i+1)\tau}X(\delta_j(t))$ is almost surely finite (for $1\le j\le
k$) and $\mathcal{K}_R(t)$ can be given as
$$
\mathcal{K}_R(t) := K_R(t)\one_{\Omega_R} + \sum_{l=R+1}^{\infty}
K_l(t) \one_{\Omega^{'}_l},
$$
where
$$
\Omega_l: = \{\sup_{i\tau \le t < (i+1)\tau}|Y(t)| \le l\} \qquad
\mbox{and} \qquad \Omega^{'}_l:= \Omega_{l+1} \backslash \Omega_l
$$
for integers $l\ge 1$. Similarly, one proves that ($D_3$) is a
consequence of ($C_3$). Clearly, ($D_4$) holds with
$$
\mathcal{M}:= \sum_{l=1}^{\infty}M_l\one_{\Omega^{'}_l} \in \mathcal
A
$$
where $M_l$ is from ($C_4$). Finally, Theorem \ref{SDEwRC_theorem}
is used here so as to obtain a unique solution on
$[i\tau,\,(i+1)\tau)$. Then, one observes that
$$
X((i+1)\tau):= X(i\tau) + \int_{i\tau}^{(i+1)\tau} \beta(t,Y(t),
X(t))dt + \int_{i\tau}^{(i+1)\tau}\alpha(t, Y(t), X(t))dW_t,
$$
is well-defined, and that concludes the induction, and consequently,
the proof is complete. $\phantom0$ \hfill  $\Box$

\section{Convergence in Probability}

For each integer $n\geq1$, let $b_n=b_n(t,x)$ and
$\sigma_n=\sigma_n(t,x)$ be $\mathcal P\otimes\mathcal B(\mathbb
R^d)$-measurable functions  with values in $\mathbb R^d$ and
$\mathbb R^{d\times m}$ respectively. Let $t_0$ and $t_1$ be
positive constants such that $0 \le t_0<t_1 \le T$. Moreover,
consider the following Euler scheme
\begin{equation}    \label{Euler}
dX_n(t)= b_n(t, X_n(\kappa_n(t)))\,dt +
\sigma_n(t,X_n(\kappa_n(t)))\,dW_t, \qquad \forall t\in[t_0,\,t_1],
\end{equation}
where $X_n(t_0)=X_{n0}$ is an $\mathcal{F}_{t_0}$-measurable random
variable and $\kappa_n(t):=[nt]/n$.

\noindent In order to prove Theorem \ref{TheoremEulerSDDE}, first we
present a slight generalisation of a result from \cite{Krylov1990}
on Euler approximations of stochastic differential equations.

\begin{theorem}  \label{main}
Consider the Euler scheme (\ref{Euler}) for equation (\ref{random}).
Let conditions ($D_1$)-($D_4$) hold. Moreover, assume that for every
$R>0$ there exists $L_{nR} \in \mathcal A$ such that
\begin{equation}\label{difference}
\sup_{|x|\le R}[|b_n(t,x)- b(t,x)|+ |\sigma_n(t,x)-\sigma(t,x)|^2]
\le L_{nR}(t) \quad \mbox{(a.s.)}
\end{equation}
and
$$
\int_0^T L_{nR}(t) dt \xrightarrow[]{\mathbb{P}}  0 \qquad \mbox{as
} n\to \infty.
$$
Finally, let $X_{n0} \xrightarrow[]{\mathbb{P}}  X(t_0)$ as $n\to
\infty$. Then
$$
\sup_{t_0 \le t\le t_1}|X_n(t)-X(t)| \xrightarrow[]{\mathbb{P}}  0
\qquad \mbox{as } n\to \infty.
$$

\end{theorem}

\begin{proof}
One observes first that conditions ($D_2$) and ($D_4$) together with
(\ref{difference}) imply that for every $R>0$, there exists $N_R(t)
\in \mathcal{A}$ such that
\begin{equation}\label{integrability}
\sup_{|x|\le R}[|b_n(t,x)|+ |\sigma_n(t,x)|^2] \le N_R(t) \quad
\mbox{(a.s.)}
\end{equation}
for every $t\in[t_0,\,t_1]$. By introducing
$$
H_n(t) := X_{n0}+ \int_{t_0}^t b(u, X_n(\kappa_n(u)))du +
\int_{t_0}^t \sigma(u, X_n(\kappa_n(u)))dW_u
$$
and $p_n(t):= X_n(\kappa_n(t)) - H_n(t)$, one obtains that
\begin{equation}\label{H}
H_n(t) = H_n(t_0) + \int_{t_0}^t b(u, H_n(u)+p_n(u))du +
\int_{t_0}^t \sigma(u, H_n(u)+p_n(u))dW_u
\end{equation}
with $H_n(t_0) = X_{n0}$. Furthermore, by introducing $e_n(t):=
X_n(t)-H_n(t)$ and the following stopping time
$$
\tau_n(R):=\inf\{t\ge t_0: |H_n(t)|+|e_n(t)|\ge R/2\}
$$
for any $R>0$, one observes that
$$
|X_n(t)|\le R/2 \qquad \mbox{and} \qquad
|p_n(t)|=|X_n(\kappa_n(t))-H_n(t)|\le R \qquad \mbox{on }
(t_0,\tau_n(R)].
$$
In addition, one calculates that as $n\to\infty$,
\begin{equation} \label{**}
\mathbb{P}(\tau_n(R)\le t_1, \sup_{t_0 \le t\le
\tau_n(R)}|H_n(t)|\le \tfrac{R}{4}) \le \mathbb{P}(\sup_{t_0 \le
t\le \tau_n(R)\wedge t_1}|e_n(t)| \ge \tfrac{R}{4} )\to 0,
\end{equation}
due to (\ref{difference}) and known results on convergence of
stochastic integrals. Moreover, one observes that $p_n(t) =
X_n(\kappa_n(t))-X_n(t) + e_n(t)$, and thus
\begin{equation} \label{*}
\mathbb{E}\int_{t_0}^{t_1 \wedge \tau_n(R)}|p_n(t)|dt  \le
\mathbb{E}\int_{t_0}^{t_1 \wedge \tau_n(R)}|X_n(\kappa_n(t)) -
X_n(t)|dt + \mathbb{E}\int_{t_0}^{t_1 \wedge \tau_n(R)}| e_n(t)|dt.
\end{equation}
By taking into account property (\ref{integrability}) and Lebesgue's
dominated convergence theorem, one  concludes that
\begin{align}
|X_n(\kappa_n(t)) - X_n(t)| \one_{[t_0,\, \tau_n(R)\wedge t_1]} \le
& \one_{[t_0,\, \tau_n(R)\wedge t_1]}\Big|\int_{\kappa_n(t)}^t
b_n(u, X_n(\kappa_n(u)))du \Big| \nonumber \\ & +\one_{[t_0,\,
\tau_n(R)\wedge t_1]}\Big|\int_{\kappa_n(t)}^t \sigma_n(u,
X_n(\kappa_n(u)))dW_u \Big|
\end{align}
converges to 0 in probability for each $t$, since one observes that
\begin{align}
\one_{[t_0,\, \tau_n(R)\wedge t_1]}\int_{\kappa_n(t)}^t |\sigma_n(u,
X_n(\kappa_n(u)))|^2du & \le \int_{\kappa_n(t)\wedge \tau_n(R)\wedge
t_0 }^{t\wedge \tau_n(R)\wedge t_1} |\sigma_n(u,
X_n(\kappa_n(u)))|^2du \nonumber \\ & = \int_{t_0}^{t_1}
\one_{A_n}|\sigma_n(u, X_n(\kappa_n(u)))|^2du \nonumber
\end{align}
converges almost surely to 0 as $n\to \infty$, where
$A_n:=(\kappa_n(t)\wedge \tau_n(R),\,t\wedge \tau_n(R) ] $. Hence
$$
\lim_{n\to\infty}\mathbb{E}\int_{t_0}^{t_1 \wedge
\tau_n(R)}|X_n(\kappa_n(t)) - X_n(t)|dt =
\lim_{n\to\infty}\int_{t_0}^{t_1}\mathbb{E} \one_{(t_0,\,
\tau_n(R)]}|X_n(\kappa_n(t)) - X_n(t)|dt  =  0,
$$
by Lebesgue's dominated convergence theorem, since
$$
|X_n(\kappa_n(t)) - X_n(t)| \le R \qquad \mbox{on }
(t_0,\,\tau_n(R)].
$$
Due to equation (\ref{difference}) and the application of the
dominated convergence theorem
$$
\lim_{n\to\infty}\mathbb{E}\int_{t_0}^{t_1 \wedge \tau_n(R)}|
e_n(t)|dt =0,
$$
which results in
$$
\lim_{n\to\infty}\mathbb{E}\int_{t_0}^{t_1 \wedge
\tau_n(R)}|p_n(t)|dt =0.
$$
Thus, the corresponding conditions of Lemma 2 in Krylov
\cite{Krylov1990} are satisfied and, therefore,
$$
\sup_{t_0 \le t\le t_1}|H_n(t)-H(t)| \xrightarrow[]{\mathbb{P}}  0,
\qquad \mbox{as } n\to \infty,
$$
for some process $\{H(t)\}_{t\in[t_0,\, t_1]}$. Furthermore, one
calculates that for any $\epsilon>0$,
\begin{equation} \label{error}
\mathbb{P}(\sup_{t_0 \le t\le t_1}|e_n(t)|\ge \epsilon )  \le
\mathbb{P}(\sup_{t_0 \le t\le \tau_n(R)\wedge t_1}|e_n(t)|\ge
\epsilon) + \mathbb{P}(\tau_n(R)\le t_1).
\end{equation}
Moreover,
\begin{align}
\mathbb{P}(\tau_n(R)\le t_1) & \le \mathbb{P}(\sup_{t_0 \le t\le
\tau_n(R)\wedge t_1}\{|H_n(t)|+|e_n(t)|\}\ge \tfrac{R}{2}) \nonumber \\
& \le   \mathbb{P}(\sup_{t_0 \le t\le t_1}|H_n(t)|\ge \tfrac{R}{4})
+ \mathbb{P}(\sup_{t_0 \le t\le \tau_n(R)\wedge t_1}|e_n(t)| \ge
\tfrac{R}{4} ), \nonumber
\end{align}
which implies, by taking into account (\ref{**})
\begin{equation} \label{countableR}
\limsup_{n\to\infty}\mathbb{P}(\tau_n(R)\le t_1) \le
\limsup_{n\to\infty}\mathbb{P}(\sup_{t_0 \le t\le t_1}|H_n(t)|\ge
\tfrac{R}{4}) = \mathbb{P}(\sup_{t_0 \le t\le t_1}|H(t)|\ge
\tfrac{R}{4})
\end{equation}
for all $R>0$ apart from countably many. Letting $R=R_k\to\infty$,
for points $R_k$ where (\ref{countableR}) holds, one obtains
$$
\lim_{R_k\to\infty}\limsup_{n\to\infty}\mathbb{P}(\tau_n(R)\le
t_1)=0.
$$
Thus, by letting $n \to \infty$ and then $R\uparrow\infty$ in
(\ref{error}), one further obtains that
$$
\lim_{n\to\infty}\mathbb{P}(\sup_{t_0 \le t\le t_1}|e_n(t)|\ge
\epsilon ) =0.
$$
As a result,
\begin{equation} \label{X_n_to_H}
\sup_{t_0 \le t\le t_1}|X_n(t)-H(t)|\le \sup_{t_0 \le t\le
t_1}|e_n(t)| + \sup_{t_0 \le t\le t_1}|H_n(t)-H(t)|
\xrightarrow[]{\mathbb{P}} 0, \qquad \mbox{as } n\to \infty.
\end{equation}
Furthermore
$$
\int_{t_0}^{t_1} |b_n(u, X_n(\kappa_n(u))) - b(u,H(u))|du
\xrightarrow[]{\mathbb{P}}  0, \qquad \mbox{as } n\to \infty,
$$
since
\begin{equation} \label{b_n_to_b}
\int_{t_0}^{t_1} |b_n(u, X_n(\kappa_n(u))) - b(u,
X_n(\kappa_n(u)))|du \xrightarrow[]{\mathbb{P}}  0, \qquad \mbox{as
} n\to \infty,
\end{equation}
due to (\ref{difference}), and
$$
\int_{t_0}^{t_1} |b(u, X_n(\kappa_n(u))) -
b(u,H(u))|du\xrightarrow[]{\mathbb{P}}  0, \qquad \mbox{as } n\to
\infty,
$$
due to the continuity of $b(t,x)$ in $x$, (\ref{X_n_to_H}), ($D_2$)
and the application of Lebesgue's dominated convergence theorem.
More precisely, equation (\ref{b_n_to_b}) holds since
\begin{align}
\mathbb{P}(\int_{t_0}^{t_1} |b_n(u, X_n(\kappa_n(u))) & -
b(u,X_n(\kappa_n(u))|du > \epsilon) \le \mathbb{P}(\rho_n(R) \le
t_1) + \nonumber
\\& \mathbb{P}(\int_{t_0}^{t_1\wedge\rho_n(R)}  |b_n(u, X_n(\kappa_n(u))) -
b(u,X_n(\kappa_n(u))|du> \epsilon) \nonumber
\end{align}
for any $\epsilon>0$ and $R>0$, where
$$
\rho_n(R) :=\inf\{t\ge t_0: |X_n(t)|\ge R\}.
$$
One then observes that due to (\ref{difference}),
\begin{align}
\mathbb{P}(\int_{t_0}^{t_1\wedge\rho_n(R)}  |b_n(u,
X_n(\kappa_n(u))) & - b(u,X_n(\kappa_n(u))|du> \epsilon) \le
\nonumber
\\ \mathbb{P}(\int_{t_0}^{t_1} \sup_{|x|\le R} |b_n(u, x) & -
b(u,x)|du> \epsilon) \to 0, \quad \mbox{as } n\to\infty. \nonumber
\end{align}
Moreover,
\begin{align}
\mathbb{P}(\rho_n(R) \le t_1) & \le \mathbb{P}(\sup_{t_0 \le t\le t_1}|X_n(t)|\ge R) \nonumber \\
& \le \mathbb{P}(\sup_{t_0 \le t\le t_1}|X_n(t) - H(t)|\ge
\tfrac{R}{2})  + \mathbb{P}(\sup_{t_0 \le t\le t_1}|H(t)|\ge
\tfrac{R}{2}) \nonumber
\end{align}
which yields, due to (\ref{X_n_to_H}), that
$$
\lim_{R\to\infty}\limsup_{n\to\infty}\mathbb{P}(\rho_n(R) \le t_1)
=0.
$$
One similarly proves that
$$
\int_{t_0}^{t_1} |\sigma_n(u, X_n(\kappa_n(u))) -
\sigma(u,H(u))|^2du\xrightarrow[]{\mathbb{P}}  0, \qquad \mbox{as }
n\to \infty.
$$
In other words, the Euler scheme converges in probability to $H(t)$,
uniformly in $t\in[t_0, \,t_1]$, and $H(t)$ satisfies
$$
dH(t) = b(t, H(t))dt + \sigma(t, H(t))dW_t, \mbox{ } \forall
t\in[t_0, \,t_1],
$$
which yields $H(t)=X(t)$ (a.s.) for every $t\in[t_0, \,t_1]$. The
proof is complete. \hfill $\Box$
\end{proof}

\noindent We are ready now to proceed with the proof of the main
result of this section.

\noindent {\it Proof of Theorem \ref{TheoremEulerSDDE}.} We prove
the theorem by showing that
\begin{equation}   \label{aim_proof2}
\sup_{(i-1)\tau \le t\le
i\tau}|X_n(t)-X(t)|\xrightarrow[]{\mathbb{P}}  0 \qquad \mbox{as }
n\to \infty,
\end{equation}
for every $i\in \{1,\ldots, N\}$. For $i=1$, the problem reduces to
the well-known case of stochastic differential equations (without
delay) where the assumptions ($C_1$)-($C_5$) are enough to prove the
result. Furthermore, assume (\ref{aim_proof2}) is true for $i< N$.
Then, as noted in the proof of Theorem \ref{SDDE theorem},
conditions ($C_1$)--($C_4$) imply that ($D_1$)--($D_4$) hold for $b$
and $\sigma$ as defined in (\ref{random_b_sigma}) with $t_0=i\tau$
and $t_1=(i+1)\tau$. Moreover, consider
\begin{equation} \label{random_b_n_sigma_n}
b_n(t,x):=\beta(t, Y_n(t), x), \quad\sigma_n(t,x):=\alpha(t, Y_n(t),
x).
\end{equation}
where $Y_n(t):=(X_n(\delta_1(t)),\ldots, X_n(\delta_k(t)))$. Then,
condition ($C_5$) implies that (\ref{difference}) holds since for
every $t\in [i\tau,\,(i+1)\tau)$ and $R>0$,
$$
\sup_{|x|\le R}[|\beta(t, Y_n(t), x) - \beta(t, Y(t), x)| +
|\alpha(t, Y_n(t), x) - \alpha(t,  Y(t),
x)|^2]\xrightarrow[]{\mathbb{P}}  0 \quad \mbox{as } n\to \infty.
$$
Thus, the application of Lebesgue's dominated convergence theorem,
due to ($C_2$) and Remark \ref{remark}, yields that, for every
$R>0$,
$$
\mathbb{P}(\int_{i\tau}^{(i+1)\tau} \sup_{|x|\le R}[|\beta(t,
Y_n(t), x) - \beta(t, Y(t), x)| + |\alpha(t, Y_n(t), x) - \alpha(t,
Y(t), x)|^2] dt> \epsilon) \to 0,
$$
as $n\to \infty$. Consequently, in light of Theorem \ref{main}, one
concludes that the inductive step is correct and thus the desired
result is obtained. \phantom0 \hfill $\Box$

\section{Rate of convergence}

Let $\{X_n\}_{n\ge 1}$  be a sequence of almost surely finite random
variables and $\{b_n\}_{n=1}^{\infty}$ be a positive numerical
sequence. Then
$$
X_n=o(b_n)
$$
denotes that there exists a sequence of random variables
$\{\eta_n\}_{n\ge 1}$ converging to 0 almost surely, such that
$$
|X_n| \le \eta_n b_n \qquad \text{(a.s.) for any $n\ge 1$}.
$$
A useful lemma follows that originates from the Gy\"{o}ngy and
Krylov \cite{Krylov_Istvan1980}.

\begin{lemma}                                                                     \label{Gyongy-Shmatkov}
Let $X_n:=\{X_n(t)\}_{t\in[0,T]}$ be a cadlag stochastic process
taking values in $\mathbb{R}^k$ for every integer $n\ge1$. Define
$$
\tau_{n\epsilon}=\inf\{t\in[0,T]: |X_n(t)|\ge \epsilon\}, \qquad
X_{n\epsilon}(t) = X_n(t\wedge\tau_{n\epsilon})
$$
for some $\epsilon>0$. Then, the following statements hold:
\begin{itemize}
\item[(i)] If $\sup_{t\in[0,T]}|X_{n\epsilon}(t)| \to 0$ in
probability, then $\sup_{t\in[0,T]}|X_{n}(t)| \to 0$ in probability
as well.
\item[(ii)] If $\sup_{t\in[0,T]}|X_{n\epsilon}(t)| \to 0$ almost surely,
then $\sup_{t\in[0,T]}|X_{n}(t)| \to 0$ almost surely as well.
\item[(iii)] If $\sup_{t\in[0,T]}|X_{n\epsilon}(t)|= \mathcal{O}(a_n)$
for a numerical sequence $0<a_n \to 0$, then
$\sup_{t\in[0,T]}|X_{n}(t)|= \mathcal{O}(a_n)$ as well.
\end{itemize}
\end{lemma}

\begin{proof}
See Lemma 3.5 in Gy\"{o}ngy and Shmatkov \cite{Istvan_Shmatkov}.
\hfill $\Box$
\end{proof}

\begin{lemma} \label{Gyongy-Krylov lemma} Let $T\in[0,\,\infty)$ and let $f:=\{f_t\}_{t\in[0,T]}$ and
$g:=\{g_t\}_{t\in[0,T]}$ be non-negative continuous
$\mathbb{F}$-adapted processes such that, for any constant $c>0$,
$$
\mathbb{E}[f_{\tau}\one_{\{g_0\le c\}}] \le
\mathbb{E}[g_{\tau}\one_{\{g_0\le c\}}]
$$
for any stopping time $\tau\le T$. Then, for any stopping time
$\tau\le T$ and $\gamma\in(0,1)$,
$$
\mathbb{E}[\sup_{t\le\tau}f_t^{\gamma}] \le
\frac{2-\gamma}{1-\gamma} \mathbb{E}[\sup_{t\le\tau}g_t^{\gamma}]
$$

\end{lemma}
\begin{proof}
See \cite{Krylov_book} and also Gy\"{o}ngy and Krylov
\cite{Krylov_Istvan2003}. \hfill $\Box$
\end{proof}

\noindent The proof of the following lemma is an easy exercise left
for the reader.

\begin{lemma}                                                          \label{easy}
Let $X_n=\{X_n(t)\}_{t\in[0,T]}$ be a cadlag stochastic process
taking values in $\mathbb{R}^k$ for every integer $n\ge1$, and let
$\{a_n\}_{n=1}^{\infty}$ be a positive numerical sequence. Assume
there exists a sequence of stopping times
$\{\tau_R\}_{R=1}^{\infty}$, such that
$\lim_{R\to\infty}P(\tau_R<T)=0$, and for each $R$
$$
\sup_{t\le T}|X_n(t\wedge\tau_R)|= \mathcal{O}(a_n).
$$
Then
$$
\sup_{t\le T}|X_n(t)| = \mathcal{O}(a_n).
$$
\end{lemma}

\noindent To formulate our next lemma we consider for each integer
$n\geq1$ an It\^o process $Z_n=\{Z_n(t)\}_{t\in[0,T]}$ with
stochastic differential
$$
dZ_n(t) = f_n(t)dt + g_n(t)dW_t, \qquad t\in [0,\,T],
$$
where $f_n$ and $g_n$ are adapted stochastic processes with values
in $\mathbb{R}^d$ and $\mathbb{R}^{d\times m}$  respectively, such
that almost surely
$$
\int_0^T|f_n(t)|dt<\infty, \qquad \int_0^T|g_n(t)|^2dt<\infty.
$$
\begin{lemma}                                                              \label{lemma 1.1.9.11}
Let $\gamma>0$ be a fixed number and assume that
\begin{equation}                                                                 \label{2.1.9.11}
Z_n(0)= \mathcal{O}(n^{-\kappa})\quad \text{for all
$\kappa<\gamma$},
\end{equation}
and almost surely
\begin{equation} \label{monotonicitylocLipdiff}
\max (Z_n(t)f_n(t), |g_n(t)|^2) \le L_n(t)|Z_n(t)|^2 + \eta_n(t),
\qquad \qquad \mbox{for all } \, t\in [0,\,T],
\end{equation}
where $L_n$ and $\eta_n$ are non-negative adapted processes such
that
\begin{equation}                                                                     \label{conditionL}
\int_0^T L_n(t)dt =o(\ln n)
\end{equation}
and
\begin{equation}                                                                     \label{condition_eta}
\int_0^T \eta_n(t)dt = \mathcal{O}(n^{-2\kappa})\quad \text{for any
$\kappa<\gamma$}.
\end{equation}
Then,
\begin{equation}                                                                        \label{5.1.10.11}
\sup_{t\le T}|Z_n(t)| = \mathcal{O}(n^{-\kappa}) \quad\text{for any
$\kappa<\gamma$.}
\end{equation}
\end{lemma}

\begin{proof}
Let $\kappa \in (0, \,\gamma)$ and set $\Omega_R=\{\sup_{n\ge
1}|Z_n(0)|n^{\kappa} \leq R \}$. Note that $\Omega_R$ is
$\mathcal{F}_0$ measurable. Then $\lim_{R\to\infty}P(\Omega_R)=1$ by
condition \eqref{2.1.9.11}, i.e., it is enough to prove
\eqref{5.1.10.11} for almost every $\omega\in\Omega_R$ for each $R$.
Thus by replacing $Z_n$, $f_n$ and $g_n$ with
$R^{-1}\one_{\Omega_R}Z_n$, $R^{-1}\one_{\Omega_R}f_n$ and
$R^{-1}\one_{\Omega_R}g_n$, respectively, we see that without loss
of generality we may assume that almost surely
\begin{equation}                                                                           \label{6.1.10.11}
|Z_n(0)|\leq n^{-\kappa}\quad\text{for all $n\geq1$}.
\end{equation}
Using this assumption we consider the stopping time $\tau_n
:=\inf\{t\ge0: |Z_n(t)|\geq1\}$ to obtain that
\begin{equation}
\sup_{t\leq T}|Z_n(t\wedge\tau_n)| \leq 1, \qquad \text{for all
$n\ge1$}. \nonumber
\end{equation}
Thus, by Lemma \ref{Gyongy-Shmatkov}, replacing $Z_n$, $f_n$, $g_n$,
$L_n$ and $\eta_n$ with $Z_n(\cdot\wedge\tau_n)$,
$f_n\one_{[0,\,\tau_n]}$, $g_n\one_{[0,\,\tau_n]}$,
$L_n\one_{[0,\,\tau_n]}$ and $\eta_n\one_{[0,\,\tau_n]}$
respectively, without loss of generality we may assume
\begin{equation}                                                                           \label{condition_Z}
\sup_{t\leq T}|Z_n(t)| \leq 1, \qquad \text{for all $n\geq1$}.
\end{equation}
Consider for every integer $R\geq1$ the stopping time
$$
\sigma_R :=\inf\{t\ge0: \sup_{n\ge
1}n^{2\kappa}\int_0^t\eta_n(s)ds\ge R^2\}\wedge T
$$
to get
\begin{equation*}
\int_0^{\sigma_R}\eta_n(s)ds \le R^2 n^{-2\kappa}, \qquad \mbox{for
all } n\ge 1.
\end{equation*}
Due to condition \eqref{condition_eta} we have
$\lim_{R\to\infty}P(\sigma_R<T)=0$. Hence by virtue of Lemma
\ref{easy} we need only show \eqref {5.1.10.11} for
$Z_n(\cdot\wedge\sigma_R)$, for each $R$, in place of $Z_n$. Thus
using $R^{-1}Z_n(\cdot\wedge\sigma_R)$,
$R^{-1}f_n\one_{[0,\,\sigma_R]}$, $R^{-1}g_n\one_{[0,\,\sigma_R]}$,
$L_n\one_{[0,\,\sigma_R]}$ and $R^{-2}\eta_n\one_{[0,\,\sigma_R]}$
in place of $Z_n$, $f_n$, $g_n$, $L_n$ and $\eta_n$ respectively,
without loss of generality we may assume
\begin{equation}                                                                        \label{condition_eta2}
\int_0^{T}\eta_n(s)ds \leq n^{-2\kappa}, \qquad \mbox{for all } n\ge
1.
\end{equation}
Introduce finally the stopping times
$$
\rho_N^{\epsilon} :=\inf\{t\ge0: \sup_{n\ge
N}\frac{\int_0^tL_n(s)ds}{\ln n}\ge \epsilon\}\wedge T,
$$
for integers $N\ge2$ and for any (small) $\epsilon>0$. Then
$$
\exp\Big(\int_0^{\rho_N^{\epsilon}}L_n(s)ds\Big) \leq n^{\epsilon},
\qquad \mbox{for all } n\ge N,
$$
which implies that the random variable
$$
\psi^N_{\epsilon}:= \sup_{n\ge
1}n^{-\epsilon}\exp\Big(\int_0^{\rho_N^{\epsilon}}L_n(s)ds\Big)
$$
is almost surely finite and we have
\begin{equation*}
\exp\Big(\int_0^{\rho_N^{\epsilon}}L_n(s)ds\Big) \le
\psi^N_{\epsilon} n^{\epsilon}, \qquad \mbox{for all } n\ge 1.
\end{equation*}
Due to condition \eqref{conditionL}
$\lim_{N\to\infty}P(\rho_N^{\epsilon}<T)=0$. Thus using Lemma
\ref{easy} as before,  we can see that without loss of generality we
may assume that for any small $\varepsilon>0$ there is a finite
random variable $\psi_{\varepsilon}$ such that almost surely
\begin{equation}                                                                             \label{condition_L2}
\exp\Big(\int_0^{T}L_n(s)ds\Big)
 \leq
\psi_{\epsilon} n^{\epsilon}, \qquad \mbox{for all } n\ge 1
\end{equation}
holds. Now we prove the lemma under the additional conditions
\eqref{6.1.10.11} through \eqref{condition_L2}. Set
$$
\phi_n(t):=\exp\big(-(2r+1)\int_0^tL_n(s)ds\big).
$$
Then, for any $r\ge 2$, Ito's formula yields
\begin{align}
d\big(\phi_n(t)|Z_n(t)|^2\big)^r = &
2r\phi^r_n(t)\Big(|Z_n(t)|^{2(r-1)}Z_n(t)dZ_n(t)+
(r-1)|Z_n(t)|^{2(r-2)}|g^T_n(t)Z_n(t)|^2dt \nonumber \\ & +
\frac{1}{2}|Z_n(t)|^{2(r-1)}|g_n(t)|^2dt\Big) +
|Z_n(t)|^{2r}d\phi^r_n(t) \nonumber \\
\le & r\phi^r_n(t)|Z_n(t)|^{2(r-1)} \Big(2Z_n(t)dZ_n(t)+
(2r-1)|g_n(t)|^2dt\Big)  \nonumber \\ & -
r(2r+1)L_n(t)\phi^r_n(t)|Z_n(t)|^{2r}dt \nonumber.
\end{align}
Hence by \eqref{monotonicitylocLipdiff},
$$
d\big(\phi_n(t)|Z_n(t)|^2\big)^r \le
r(2r+1)\phi^r_n(t)|Z_n(t)|^{2(r-1)}\eta_n(t)dt +
2r|Z_n(t)|^{2(r-1)}\phi^r_n(t)Z_n(t)g_n(t)dW_t
$$
on $[0,\,T]$. Thus, for every stopping time $\tau\le T$
$$
\mathbb{E}\big(\phi_n(\tau)|Z_n(\tau)|^2\big)^r \leq
n^{-2r\kappa}+r(2r+1) \mathbb{E}\int_0^{\tau}
\big(\phi_n(t)|Z_n(t)|^2\big)^{r-1}\eta_n(t)dt
$$
Then, one applies Lemma \ref{Gyongy-Krylov lemma} with the
non-negative processes $f$ and $g$ being represented by $f_t:=
\big(\phi_n(t)|Z_n(t)|^2\big)^r$ and
$$
g_t:=  n^{-2r\delta\kappa} + r(2r+1)\int_0^t
\phi^r_n(s)|Z_n(s)|^{2(r-1)}\eta_n(s)ds,
$$
for every $t\in[0,\,T]$, to obtain that, for any $\delta\in
(0,\,1)$,
$$
\mathbb{E}[\sup_{t\le T}\big(\phi_n(t)|Z_n(t)|^2\big)^{r\delta}] \le
Cn^{-2r\delta\kappa}+ C \mathbb{E}\big(\int_0^T
\big(\phi_n(t)|Z_n(t)|^2\big)^{r-1}\eta_n(t)dt\big)^{\delta}.
$$
Hence the application of Young's inequality yields
$$
\mathbb{E}[\sup_{t\le T}\big(\phi_n(t)|Z_n(t)|^2\big)^{r\delta}]
\leq C n^{-2r\delta\kappa}+ \frac{1}{2}\mathbb{E}\sup_{t\le
T}\big(\phi_n(t)|Z_n(t)|^2\big)^{r\delta} + C
\mathbb{E}\Big(\int_0^T\eta_n(t)dt\Big)^{r\delta}.
$$
Thus, due to \eqref{condition_eta2} we have
$$
\sum_n \mathbb{P}\Big(\sup_{t\le
T}\big(\phi_n(t)|Z_n(t)|^2\big)>n^{-2\kappa^{'}}\Big) \le C\sum_n
n^{2r\delta(\kappa^{'}- \kappa)}<\infty
$$
for a sufficiently large $r$, $\delta \in(0,1)$ and any
$\kappa^{'}<\kappa<\gamma$. Here and above, C denotes constants that
depend on $r$ and $\delta$ but are independent of $n$.
Borel-Cantelli lemma then implies that
 for each $\kappa<\gamma$ there is a finite random variable
$\zeta_{\kappa}$ such that almost surely
$$
\sup_{t\le T}\big(\phi_n(t)|Z_n(t)|^2\big) \leq \zeta_{\kappa}
n^{-2\kappa}
$$
for all $n\geq1$. Hence, due to \eqref{condition_L2}
$$
\sup_{t\le T}|Z_n(t)|^{2}
 \leq
 \zeta_{\kappa}n^{-2\kappa}
 \exp\big((2r+1)\int_0^TL_n(t)dt \big)
 \leq
 \psi_{\epsilon}
\zeta_{\kappa}n^{\epsilon(2r+1)-2\kappa}.
$$
By taking $\epsilon$ sufficiently small and $\kappa$ sufficiently
close to $\gamma$, the desired result follows immediately. The proof
is complete. \hfill $\Box$
\end{proof}

\noindent One is then ready to proceed with the calculation of the
rate of convergence for the Euler scheme (\ref{Euler}). Key
conditions are described below.

\begin{itemize}
\item[($E_1$)] For every $R>0$, there exist finite
$\mathcal{F}_{t_0}$-measurable random variable $C_R$ such that,
almost surely,
\begin{equation}                                                            \label{locLip1}
|b(t,x)-b(t,y)|\le C_R|x-y|
\end{equation}
\begin{equation}                                                           % \label{locLip2}
|\sigma(t,x)-\sigma(t,y)|^2\le C_R|x-y|^2
 \nonumber
\end{equation}
for every $t\in[t_0,\,t_1]$ and $|x|,\,|y| < R$.
\item[($E_2$)] For every $R>0$, there exist finite
$\mathcal{F}_{t_0}$-measurable random variables $K_R$ such that,
almost surely,
\begin{equation}                                                            %\label{locbdd}
|b_n(t,x)|\le K_R \qquad \mbox{and} \qquad |\sigma_n(t,x)|\le K_R
\nonumber
\end{equation}
for every $n\ge 1$, $t\in[t_0,\,t_1]$ and $|x| < R$.
\item[($E_3$)] For every $R>0$, there exist adapted processes $M_{Rn}$ such that,
almost surely,
\begin{equation}                                                            %\label{difference2}
|b(t,x) - b_n(t,x)|^2 \le M_{Rn}(t)  \nonumber
\end{equation}
\begin{equation}                                                            %\label{difference3}
|\sigma(t,x) - \sigma_n(t,x)|^2 \le M_{Rn}(t) \nonumber
\end{equation}
for every $n\ge 1$, $t\in[t_0,\,t_1]$ and $|x| < R$, and for every
$\gamma<1/2$,
$$
\int_{t_0}^{t_1} M_{Rn}(t) dt =\mathcal{O}(n^{-2\gamma}). \nonumber
$$
\item[($E_4$)] Alternatively to (\ref{locLip1}), there exists a finite $\mathcal{F}_{t_0}$-measurable random variable $C_R$ such that,
almost surely,
\begin{equation} %\label{MonotonicityEulerRandom}
2(x-y)\Big(b(t,x)-b(t,y)\Big) \le C_R |x-y|^2 \nonumber
\end{equation}
for every $t\in[t_0,\,t_1]$ and $|x|,\,|y| < R$.
\end{itemize}
\begin{remark} \label{Remark2}
Conditions ($E_2$) and ($E_3$) imply that for every $R>0$, there
exists a process $M_R \in \mathcal{A}$ such that, almost surely,
$$
|b(t,x)|^2 \le M_R(t) \qquad \mbox{and} \qquad |\sigma(t,x)|^2 \le
M_R(t).
$$
for any $t\in[t_0,\,t_1]$ and $|x| < R$. Hence, due to conditions
($E_1$) -- ($E_3$), equation (\ref{random}) has a unique (complete)
local solution. The same is true if (\ref{locLip1}) is replaced by
($E_4$). The existence of a unique (global) solution can be
guaranteed by appropriate assumptions on the growth of $b$ and
$\sigma$, e.g. by ($D_4$).
\end{remark}
\begin{theorem} \label{TheoremEulerRateRandom}
Let conditions ($E_1$)--($E_3$) hold. Assume that equation
(\ref{random}) with initial data $X(t_0)$ admits a solution
$\{X(t)\}_{t\in[t_0,\, t_1]}$. Let $\{X_n(t)\}_{t\in[t_0,\, t_1]}$
denote the solution of the Euler scheme (\ref{Euler}) with initial
data $X_{n0}$ such that
\begin{equation}                                                            \label{initial data}
|X(t_0) - X_{n0}|=\mathcal{O}(n^{-\gamma}) \quad \mbox{(a.s.)}
\end{equation}
for every $\gamma<1/2$. Then
\begin{equation}                                                            \label{estimate1}
\sup_{t_0 \le t\le t_1}|X(t) - X_n(t)|=\mathcal{O}(n^{-\gamma})
\quad \mbox{(a.s.)}
\end{equation}
for every $\gamma<1/2$. Moreover, if one replaces condition
(\ref{locLip1}) with ($E_4$), then (\ref{estimate1}) holds for every
$\gamma<1/4$. In this case, it is sufficient to require that ($E_3$)
and (\ref{initial data}) hold for every $\gamma<1/4$.
\end{theorem}

\begin{proof}
Due to Lemma \ref{Gyongy-Shmatkov} and \ref{easy}, it suffices to
prove that
$$
\sup_{t_0 \le t\le t_1}|Z_n(t)|=\mathcal{O}(n^{-\gamma}) \quad
\mbox{(a.s.)}
$$
for every $\gamma<1/2$, where
$$
Z_n(t) :=  X(t\wedge \tau_{n\epsilon R}) - X_n(t\wedge
\tau_{n\epsilon R}),
$$
$$
\tau_{n\epsilon} :=\inf\{t\ge t_0: |X(t) - X_n(t)|\ge \epsilon\},
\qquad \tau_R :=\inf\{t\ge t_0: |X(t)|\ge R-1\}.
$$
and $\tau_{n\epsilon R} = \tau_{n\epsilon} \wedge \tau_R$ for every
$R>0$ and arbitrary $\epsilon \in (0,\, 1)$.

\noindent Moreover, it is enough to prove our result under the
condition that assumption ($E_2$) holds with a constant $L$ instead
of a finite $\mathcal{F}_{t_0}$-measurable random variable $K_R$.
One only needs to replace $b_n$ and $\sigma_n$ with
$b_n\one_{\{K_R\le L\}}$ and $\sigma_n\one_{\{K_R\le L\}}$
respectively, for each (fixed) $R>0$, since
$$
[K_R<L] \in \mathcal{F}_{t_0} \qquad \mbox{and} \qquad
\mathbb{P}(\bigcup_{L=1}^{\infty} [K_R<L])=1.
$$
Then, for $h_n(s):= \sigma_n(s,X_n(\kappa_n(s)))\one_{\{s\le
\tau_{n\epsilon R}\}}$, one has  $|h_n(t)|\le L$ on $[t_0,\, t_1]$
and for $r\ge2$,
\begin{align}
\mathbb{E}\Big(\int_{t_0}^{t_1}\Big|\int^t_{\kappa_n(t)}h_n(s)dW_s\Big|dt\Big)^r
& \le
K \mathbb{E}\Big(\int_{t_0}^{t_1}\Big|\int^t_{\kappa_n(t)}h_n(s)dW_s \Big|^{2r} dt\Big)^{\frac{1}{2}} \qquad \mbox{(H\"{o}lder)} \nonumber \\
& \le K \Big(\int_{t_0}^{t_1}\mathbb{E}\Big|\int^t_{\kappa_n(t)}h_n(s)dW_s \Big|^{2r} dt\Big)^{\frac{1}{2}}  \qquad \mbox {(Jensen)} \nonumber \\
& \le K n^{-\frac{r}{2}} \nonumber
\end{align}
holds, where $K$ denote positive constants which are independent of
$n$. As a result, due to Markov's inequality,
$$
\sum_n\mathbb{P}\Big(\int_{t_0}^{t_1}\Big|\int^t_{\kappa_n(t)}
h_n(s)dW_s\Big|dt>n^{-\gamma}\Big) \le K\sum_n n^{\gamma r -
\frac{r}{2}}<\infty
$$
for a sufficiently large $r$ and $\gamma\in[0,1/2)$, which proves
that
$$
\int_{t_0}^{t_1}\Big|\int^t_{\kappa_n(t)} h_n(s)dW_s\Big|dt=
\mathcal{O}(n^{-\gamma})
$$
by the application of the Borel-Cantelli lemma. Consequently,
\begin{align} \label{kappa}
\int_{t_0}^{t_1}|X_n(t) - X_n((\kappa_n(t))|\one_{\{t\le
\tau_{n\epsilon R}\}} dt & \le
\int_{t_0}^{t_1}|\int^t_{\kappa_n(t)}\one_{\{s\le \tau_{n\epsilon
R}\}}
b_n(s,X_n(\kappa_n(s)))ds|dt \nonumber \\
& + \int_{t_0}^{t_1}|\int^t_{\kappa_n(t)} h_n(s)dW_s|dt \nonumber \\
& = \mathcal{O}(n^{-\gamma})
\end{align}
for any $\gamma\in[0,1/2)$. Furthermore, in order to apply Lemma
\ref{lemma 1.1.9.11}, one defines
$$
f_n(t) := \one_{T_n}[b(t,X(t)) - b_n(t,X_n(\kappa_n(t)))]
$$
and
$$
g_n(t) := \one_{T_n}[\sigma(t,X(t))- \sigma_n(t,X_n(\kappa_n(t)))].
$$
where $T_n:=(t_0,\, \tau_{n\epsilon R}]$. Moreover, one calculates
\begin{align}
Z_n(t)f_n(t)  \le &  |Z_n(t)|\Big(|b(t,X(t)) - b(t,X_n(t))| +
|b(t,X_n(t)) - b(t,X_n(\kappa_n(t)))| \nonumber \\
& + |b(t,X_n(\kappa_n(t))) - b_n(t,X_n(\kappa_n(t)))|\Big)\one_{T_n} \nonumber \\
\le & (C_R+1)|Z_n(t)|^2  +C_R|X_n(t) - X_n(\kappa_n(t))|^2\one_{T_n}
\nonumber
\\ & + |b(t,X_n(\kappa_n(t))) - b_n(t,X_n(\kappa_n(t)))
|^2\one_{T_n}, \nonumber
\end{align}
and,
\begin{align}
|g_n(t)|^2 \le & 3\Big(|\sigma(t,X(t)) - \sigma(t,X_n(t))|^2
+ |\sigma(t,X_n(t)) - \sigma(t,X_n(\kappa_n(t)))|^2 \nonumber \\
& + |\sigma(t,X_n(\kappa_n(t)) - \sigma_n(t,X_n(\kappa_n(t)))|^2
\Big)\one_{T_n} \nonumber
\\ \le &
C_R|Z_n(t)|^2 + C_R|X_n(t) - X_n(\kappa_n(t))|^2 \one_{T_n}
\nonumber
\\  & + |\sigma(t,X_n(\kappa_n(t)) -
\sigma_n(t,X_n(\kappa_n(t)))|^2 \one_{T_n}. \nonumber
\end{align}
which yields
\begin{equation} \label{g^2_n}
\max(Z_n(t)f_n(t), |g_n(t)|^2) \le L_n|Z_n(t)|^2 + \eta_n(t), \qquad
\qquad \mbox{for all } t\in [t_0,\,t_1],
\end{equation}
due to (\ref{kappa}) and ($E_3$), where $L_n:= C_R+1$ and
\begin{align}
\eta_n(t): = \Big(C_R|X_n(t) - X_n(\kappa_n(t))|^2 & +
|b(t,X_n(\kappa_n(t))) - b_n(t,X_n(\kappa_n(t))) |^2 \nonumber \\ &
+  |\sigma(t,X_n(\kappa_n(t)) -
\sigma_n(t,X_n(\kappa_n(t)))|^2\Big)\one_{T_n} \nonumber
\end{align}
satisfy the conditions of Lemma \ref{lemma 1.1.9.11} for any
$\gamma\in[0,1/2)$. Thus, application of Lemma \ref{lemma 1.1.9.11}
yields the desired result (\ref{estimate1}) for any
$\gamma\in[0,1/2)$. Finally, if assumption (\ref{locLip1}) is
replaced by ($E_4$), then
\begin{align}
Z_n(t)f_n(t) = & \Big(X(t) - X_n(t)\Big)\Big[b(t,X(t)) - b_n(t,X_n(\kappa_n(t)))\Big]\one_{T_n}\nonumber \\
= & \Big(X(t) - X_n(\kappa_n(t))\Big)\Big[b(t,X(t)) - b(t,X_n(\kappa_n(t)))\Big]\one_{T_n}\nonumber \\
& + \Big(X(t) - X_n(\kappa_n(t))\Big)\Big[ b(t,X_n(\kappa_n(t))) - b_n(t,X_n(\kappa_n(t)))\Big]\one_{T_n}\nonumber \\
& + \Big(X_n(\kappa_n(t)) - X_n(t)\Big)\Big[b(t,X(t)) - b_n(t,X_n(\kappa_n(t)))\Big]\one_{T_n}\nonumber \\
\le & \frac{1}{2}C_R|X(t) - X_n(\kappa_n(t))|^2\one_{T_n}  \nonumber \\
& + |X(t) - X_n(\kappa_n(t))| |b(t,X_n(\kappa_n(t))) - b_n(t,X_n(\kappa_n(t)))|\one_{T_n}\nonumber \\
& + |X_n(\kappa_n(t)) - X_n(t)||b(t,X(t)) - b_n(t,X_n(\kappa_n(t)))|\one_{T_n}\nonumber \\
\le & (C_R+1)|Z_n(t)|^2 + (C_R+1)|X_n(\kappa_n(t)) -
X_n(t)|^2\one_{T_n} +
\frac{1}{2}M_{Rn}(t)  \nonumber \\
& + 2|X_n(\kappa_n(t)) - X_n(t)|\one_{T_n}\Big(L +
M_R^{1/2}(t)\Big), \nonumber
\end{align}
where $M_R$ comes from Remark \ref{Remark2}. Thus
\begin{equation} \label{Z_nf_n}
Z_n(t)f_n(t) \le L_n|Z_n(t)|^2 + \eta_n(t), \qquad \qquad \mbox{for
all } t\in [t_0,\,t_1],
\end{equation}
where $L_n:=C_R+1$ and $\eta_n$ is a non-negative
$\mathbb{F}$-adapted process which satisfy the conditions of Lemma
\ref{lemma 1.1.9.11} for any $\gamma\in[0,1/4)$ due to
(\ref{kappa}). Thus the proof is complete. \hfill $\Box$
\end{proof}

\begin{remark}
One could further observe that assumptions (\ref{locLip1}) \&
($E_2$) can be relaxed to allow
$$
|b(t,x)-b(t,y)|^2\le M_R(t)|x-y|^2 \quad \mbox{and} \quad
\sup_{|x|\le R}|b_n(t,x)|^2\le M_R(t),
$$
for some process $M_R \in \mathcal{A}$, while Theorem
\ref{TheoremEulerRateRandom} remains true.
\end{remark}

\noindent The proof of the main and final result of this section
follows.

\noindent {\it Proof of Theorem \ref{rateSDDEcase}.} We prove the
theorem by showing that
\begin{equation}                                                            \label{estimateSDDE1subinterval}
\sup_{(i-1)\tau \le t\le i\tau}|X(t) -
X_n(t)|=\mathcal{O}(n^{-\gamma}) \quad \mbox{(a.s.)}
\end{equation}
for every $i\in\{1,\ldots,N\}$. For $i=1$, the conditions of Theorem
\ref{TheoremEulerRateRandom} are satisfied and thus, estimate
(\ref{estimateSDDE1subinterval}) is achieved on the interval
$[0,\,\tau]$. Assume that estimate (\ref{estimateSDDE1subinterval})
holds for $i<N$. Then let us show that the conditions of Theorem
\ref{TheoremEulerRateRandom} for equations (\ref{random}) and
(\ref{Euler}) hold with initial data $X(t_0)=X(i\tau)$ and
$X_{n0}=X_n(i\tau)$ and with $b$, $\sigma$ and $b_n$, $\sigma_n$
given by (\ref{random_b_sigma}) and (\ref{random_b_n_sigma_n})
respectively. Clearly, (\ref{initial data}) is satisfied.
Furthermore, assumption ($E_1$) is satisfied since, for every $R>0$,
there exist a constant $c_R$ such that
$$
|b(t,x)-b(t,y)|\le c_R|x-y| \quad \mbox{(a.s.)}
$$
and
$$ |\sigma(t,x)-\sigma(t,y)|^2\le
c_R|x-y|^2 \quad \mbox{(a.s.)}
$$
for every $t\in[i\tau,\,(i+1)\tau)$ and $|x|$, $|y|<R$, due to
($A_2$). Moreover, for integers $l \ge 1$, define
$$
\Omega_l: = \{\sup_{n\ge 1}\sup_{m\tau \le t< (m+1)\tau}|Y_n(t)| \le
l\} \qquad \mbox{and} \qquad \Omega^{'}_l:= \Omega_l \backslash
\Omega_{l-1}
$$
and observe that $\mathbb{P}(\cup_l\Omega^{'}_l)=1$ due to
(\ref{estimateSDDE1subinterval}). Thus assumption ($E_2$) is
satisfied since
$$
\sup_{|x|\le R}|b_n(t,x)|= \sup_{|x|\le R}|\beta(t, Y_n(t), x)| \le
k_R\one_{\Omega_R} +
\sum_{l=R+1}^{\infty}k_l\one_{\Omega^{'}_l}<\infty \mbox{ (a.s.)}
$$
and
$$
\sup_{|x|\le R}|\sigma_n(t,x)|= \sup_{|x|\le R}|\alpha(t, Y_n(t),
x)| \le k_R\one_{\Omega_R} +
\sum_{l=R+1}^{\infty}k_l\one_{\Omega^{'}_l}<\infty \mbox{ (a.s.)},
$$
where $k_l$ are constants from ($A_1$). Finally, observe that
\begin{align}
\sup_{|x|\le R}|b(t,x) - b_n(t,x)|^2 & \le c_R |Y(t) - Y_n(t)|^2
\nonumber
\end{align}
and
\begin{align}
\sup_{|x|\le R}|\sigma(t,x) - \sigma_n(t,x)|^2 & \le c_R |Y(t) -
Y_n(t)|^2 \nonumber
\end{align}
for every $t\in[i\tau,\,(i+1)\tau)$, due to ($A_2$). Consequently,
($E_3$) holds for $\gamma<1/2$. The application of Theorem
\ref{TheoremEulerRateRandom} shows that the inductive step is
correct on $[i\tau,\,(i+1)\tau]$ which yields the result
$$
\sup_{t\le T}|X(t) - X_n(t)|=\mathcal{O}(n^{-\gamma}) \quad
\mbox{(a.s.)}
$$
for every $\gamma<1/2$. Moreover, if one replaces (\ref{LipDelay})
with ($A3$), then (\ref{estimateSDDE1}) holds for every $\gamma<1/4$
due to the fact that Theorem \ref{TheoremEulerRateRandom} applies
for the case where assumption (\ref{locLip1}) is replaced by
($E_4$).\hfill $\Box$

\end{document}